\documentclass[final,3p,times]{article}

\usepackage{lineno,hyperref}
\modulolinenumbers[1]

\usepackage{graphicx}
\usepackage{amssymb}
\usepackage{mathtools}
\usepackage{amsthm}
 \usepackage{amsmath}
\usepackage{dsfont}
\usepackage{epsfig}
\usepackage{float}
\usepackage{epstopdf}
\usepackage{subfigure}
\usepackage{cite}
\usepackage{xcolor}
\usepackage{latexsym,amssymb}
\usepackage{pgfplots}
\usepackage{amsmath}
\usepackage{authblk}

\usepackage{soul}

\numberwithin{equation}{section}

\theoremstyle{definition}

\newtheorem{theorem}{Theorem}[section]
\newtheorem{lemma}[theorem]{Lemma}
\newtheorem{proposition}[theorem]{Proposition}

\newtheorem{definition}[theorem]{Definition}

\theoremstyle{remark}

	\usetikzlibrary{plotmarks}
\usetikzlibrary{external}
\pgfplotsset{compat=1.3}
\newlength\figurewidth 
\newlength\figureheight
\tikzset{external/force remake=true}

\usepackage[margin=0.88in]{geometry}

\date{}

\begin{document}



\title{$h$-Trigonometric B-splines}

\author[a]{Fatma Z\"{u}rnac{\i}-Yeti\c{s}}
\author[b]{Ron Goldman}
\author[c]{Plamen Simeonov \footnote{\textbf{Email addresses:} $^a$fzurnaci@itu.edu.tr, $^b$rng@cs.rice.edu  $^c$simeonovp@uhd.edu, }}

\affil[a]{Department of Mathematics Engineering, Istanbul Technical University,  Maslak, Istanbul, 34469, Turkiye}

\affil[b]{Department of Computer Science, Rice University, Houston, Texas, 77251, USA}

\affil[c]{Department of Mathematics and Statistics, University of Houston-Downtown, Houston Texas 77002, USA}

\setcounter{Maxaffil}{0}
\renewcommand\Affilfont{\small}

\maketitle

\begin{abstract}
	We introduce discrete analogues of the exponential, sine, and cosine functions. Then using a discrete trigonometric version of a non-polynomial divided difference, we define discrete analogues of the trigonometric B-splines. We derive a two-term recurrence relation, a two-term formula for the discrete derivative, and two variants of the Marsden identity for these discrete trigonometric B-splines. Since the classical exponential, sine, and cosine functions are limiting cases of their discrete analogues, we conclude that many of the standard results for classical polynomial B-splines extend naturally both to trigonometric B-splines and to discrete trigonometric B-splines.
\end{abstract}

 \vspace{0.3cm}
 

\section{Introduction}
$h$-Trigonometric B-splines are inspired by several works of T. Lyche: in particular, his Ph.D. thesis on \textit{Discrete Polynomial Spline Approximation Methods} \cite{LycheThesis}, and  his paper on \textit{A Stable Recurrence Relation for Trigonometric B-splines} \cite{Lyche1}. Trigonometric splines go back at least to Schoenberg \cite{Schumaker2} and have been studied by many other authors: see the survey paper
\cite{Lyche survey} and the references therein. \textit{The goal of this paper is to extend the theory of B-splines from the classical trigonometric setting to
the discrete trigonometric setting}. Our approach is to invoke the $h$-quantum calculus \cite{victor} in the trigonometric setting. This work is also related to  \cite{ron-h}, where the $h$-quantum calculus is used to extend B\'{e}zier  curves  from the classical polynomial setting to the discrete polynomial setting.

We begin in Section \ref{section2} by adopting from \cite{victor} the definitions for the discrete $h$-exponential function $e^{x}_{h}$ and the discrete $h$-trigonometric functions $\sin_{h}x$ and $\cos_{h}x$. We show that the $h$-trigonometric functions $\sin_{h}x$ and $\cos_{h}x$ share many properties and identities analogous to the properties and identities of the classical trigonometric functions $\sin x$ and $\cos x$. In particular, $$e^{ix}_{h}=\cos_{h}x+i \sin_{h}x.$$
We also derive formulas relating $\sin_{h}x$ and $\cos_{h}x$ to the classical
versions of $\sin x$ and $\cos x$.

The $h$-trigonometric B-splines of degree $m$ are piecewise functions such that each piece belongs to the space. $T_{m,h}$ of all functions of the form:\\
	\begin{equation}
		f(x)= \left\{\begin{array}{ll}	a_{0}+\sum_{k=1}^{n}\left(a_{k} \cos_h(k x) +b_{k} \sin_h(k x) \right), &\mbox{if} \quad m=2 n+1 \quad(n \geqslant 0), \\ \sum_{k=1}^{n}\left( a_{k} \cos_h \left(k-\frac{1}{2}\right) x+b_{k} \sin_h \left(k-\frac{1}{2}\right) x\right), & \mbox{if} \quad m=2 n \quad(n \geqslant 1).
		\end{array}\right.
	\end{equation}
 An equivalent alternative representation for this space is that for any constant $\alpha$,
	\begin{align}
		T_{m,h}=\left\{\sum_{j=0}^{m-1} a_{j}\left(\sin_ h\left(\frac{x-\alpha}{2}\right)\right)^{j}\left(\cos_ h\left(\frac{x-\alpha}{2}\right)\right)^{m-j-1}\right\}.
\end{align}

One of our main tools for studying the $h$-trigonometric B-splines is the $h$-trigonometric version of the divided difference. In Section \ref{section3} we briefly introduce a non-polynomial divided difference, and in Section \ref{section4}  we study in more detail the properties and identities of the corresponding $h$-trigonometric divided difference.

Section \ref{section5} contains our main results on $h$-trigonometric B-splines. We begin by applying the $h$-trigonometric
divided difference to a truncated power of $\sin_{h}x$  to define the $h$-trigonometric B-splines. We derive a two-term recurrence relation for these $h$-trigonometric B-splines, analogous to the de Boor recurrence for classical polynomial B-splines. 
 We also derive a two-term formula for the $h$-derivative of these $h$-trigonometric B-splines in terms of  $h$-trigonometric B-splines of lower order.  In addition, we derive two forms of the Marsden identity for $h$-trigonometric B-splines: an $h$-exponential form and an $h$-trigonometric form. We close by deriving integral representations for the $h$-exponential divided difference and the $h$-trigonometric divided difference.

In Section \ref{section6} we briefly summarize our results. Since the classical exponential, sine, and cosine functions are limiting cases of their discrete analogues, we observe that many of the standard results for classical polynomial B-splines extend naturally both to classical trigonometric B-splines and to $h$-trigonometric B-splines. We close with a few open problems for future research.

\section{The $h$-quantum calculus}\label{section2}
In this section we provide a brief introduction to the $h$-quantum calculus, which we
shall use throughout the remainder of this paper. For further details, formulas, and
proofs, see \cite{victor}.
\subsection{The $h$-exponential and $h$-trigonometric functions}
To begin, recall that 
\begin{equation*}
	e=\lim_{n \rightarrow \infty}\left(1+\frac{1}{n}\right)^{n}=\lim_{h \rightarrow 0}(1+h)^{1/h}.
\end{equation*}
The $h$-quantum calculus the formulas of the classical calculus, but without the limits. Thus we define
\begin{align*}
e_{h}
	=(1+h)^{\frac{1}{h}}\quad \mbox{for}\quad h \neq 0 \quad \mbox{and} \quad h>-1,\\
	e^{x}_{h}
=(1+h)^{x/h}\quad \mbox{for}\quad h \neq 0 \quad \mbox{and} \quad h>-1.
\end{align*}
In particular, if $h = 1$, then $e_h = 2$ and  $e^x_{h}= 2^x$.

Replacing $x$ by $ix$ yields
\begin{align}\label{hdefexp}
	e^{ix}_{h}
	=&(1+h)^{ix/h}.
\end{align}
In analogy with $\sin x$ and $\cos x$, we define $\sin_{h}x$ and $\cos_{h}x$  by 
\begin{align}\label{cosseri}
	\cos_{h}x=\frac{e^{ix}_{h}+e^{-ix}_{h}}{2} \quad\mbox{and}\quad\sin_{h}x=\frac{e^{ix}_{h}-e^{-ix}_{h}}{2i}.
\end{align}
From these definitions, it follows immediately that
\begin{align*}
	\lim_{h\rightarrow 0} 	e^{x}_{h}=	e^{x},\quad \lim_{h \rightarrow 0}\cos_{h}x=\cos x, \quad \lim_{h \rightarrow 0}\sin_{h}x=\sin x.
\end{align*}
Thus, formulas and identities for the classical exponential, sine, and cosine functions are limiting cases of the corresponding formulas and identities for the discrete analogues of the exponential, sine, and cosine functions. Moreover, again from the definitions
\begin{align}\label{cossin1}
		&e^{ix}_{h}=(1+h)^{ix/h}=\cos_{h}x+i\sin_{h}x,\\\label{cossin2}
&e^{-ix}_{h} =(1+h)^{-ix/h}=\cos_{h}x-i\sin_{h}x.
\end{align}
	Since
\begin{align}\label{exppower}
	e^{i(x\pm y)}_{h}=e^{ix}_{h}e^{\pm i y}_{h},
\end{align} many of the classical identities for sine and cosine remain valid for the $h$-variants of sine
and cosine. In particular,
\begin{align}\nonumber
	&\cos_{h}(x\pm y)=\cos_{h}x\cos_{h}y\mp \sin_{h}x\sin_{h}y,\\ \nonumber
	&\sin_{h}(x\pm y)=\sin_{h}x\cos_{h}y\pm \cos_{h}x\sin_{h}y,\\ \nonumber
		&\sin_{h}^{2}x+\cos_{h}^{2}x=1,\\ \nonumber
	&\cos_{h}^{2}(2x)= \cos_{h}^{2}x-\sin_{h}^{2}x=1-2\sin_{h}^{2}x=2\cos_{h}^{2}x-1,\\ \nonumber
	&\sin_{h}2x=2\sin_{h}x\cos_{h}x,\\  \nonumber
		&\sin_{h}x\cos_{h}y=\frac{1}{2}\{\sin_{h}(x+y)+\sin_{h}(x-y)\},\\ \nonumber
	&\sin_{h}x\sin_{h}y=\frac{1}{2}\{\cos_{h}(x-y)-\cos_{h}(x+y)\},\\ \nonumber
		&\cos_{h}x\cos_{h}y=\frac{1}{2}\{\cos_{h}(x+y)+\cos_{h}(x-y)\}.
\end{align}
\begin{lemma}
	\begin{align}\label{identitysinexp}
		e^{ix}_{h}-e^{iy}_{h}=2i e^{i\left(x+y\right)/2}_{h}\sin_{h}\left(\left(x-y\right)/2\right),\\
		\label{identitycosexp}
		e^{ix}_{h}+e^{iy}_{h}=2 e^{i\left(x+y\right)/2}_{h}\cos_{h}\left(\left(x-y\right)/2\right).
	\end{align}
\end{lemma}
\begin{proof}
By    (\ref{exppower}) and (\ref{cosseri}),
\begin{align}
	e^{ix}_{h}-e^{iy}_{h}=e^{i\left(x+y\right)/2}_{h}\left(e^{i\left(x-y\right)/2}_{h}-e^{i\left(y-x\right)/2}_{h}\right)=2i e^{i\left(x+y\right)/2}_{h}\sin_{h}\left(\left(x-y\right)/2\right),
\end{align}
which is (\ref{identitysinexp}). The proof of (\ref{identitycosexp}) is very similar.
\end{proof}
Finally, there is a relationship between the classical sine and cosine and the $h$-variants of sine and cosine. Observe that
\begin{align*}
	e^{ix}_{h}=(1+h)^{ix/h}=e^{ix(\ln(1+h))/h}.
\end{align*}
It follows that 
\begin{align*}
	\cos_{h}x+i\sin_{h}x=\cos\left(\frac{\ln(1+h)}{h}x\right)+i\sin\left(\frac{\ln(1+h)}{h}x\right).
\end{align*}
Therefore
\begin{align*}
	\cos_{h}x=\cos\left(\frac{\ln(1+h)}{h}x\right) \quad \mbox{and}\quad \sin_{h}x=\sin\left(\frac{\ln(1+h)}{h}x\right).
	\end{align*}
Notice too that $\sin_{h}x=0$ if $x=n \pi h/\ln(1+h)$, $ n \in \mathbb{Z}$. Similarly, $\cos_{h}x=0$ if $x=(n+1/2)\pi h/\ln(1+h)$, $ n \in \mathbb{Z}$.
\subsection{The $h$-derivative and the $h$-integral}
In the $h$-quantum calculus, the classical derivative is replaced by the $h$-derivative \begin{align*}
	\Delta_{h}f(x)=\frac{f(x+h)-f(x)}{h},
\end{align*}
the formula for the derivative without the limit. The $h$-derivative satisfies an $h$-version
	of the product rule.\vspace{0.15cm}
	\textbf{$h$-Product rule:}
\begin{align}\label{derivpr}
	\Delta_{h}\left(f(x)g(x)\right)=g(x)\Delta_{h}f(x)+f(x+h)\Delta_{h}g(x).
\end{align}
Notice, however, that the chain rule does not hold for the $h$-derivative:
\begin{align*}
	\Delta_{h}e^{x}_{h}=\frac{(1+h)^{(x+h)/h}-(1+h)^{x/h}}{h}=e^{x}_{h},
\end{align*}
\begin{align}\label{1derivative_e2}
	\Delta_{h}e^{icx}_{h}=\frac{e^{ic(x+h)}_{h}-e^{icx}_{h}}{h}=\left(\frac{e^{ich}_{h}-1}{h}\right)e^{icx}_{h}.
\end{align}

The $h$-integral is an $h$-analogue of the classical integral, a
	Riemann sum without a limit. 
\begin{definition}[\textbf{$h$-Integral}]
	If \( b - a \in h\mathbb{Z} \), then  
	\[
	\int_a^b f(x)\, d_h x =
	\begin{cases}
		h \left( f(a) + f(a + h) + \cdots + f(b - h) \right) & \text{if } a < b, \\
		0 & \text{if } a = b, \\
		-h \left( f(b) + f(b + h) + \cdots + f(a - h) \right) & \text{if } a > b.
	\end{cases}
	\]	
\end{definition}
The $h$-integral satisfies an $h$-version of the
fundamental theorem of calculus as well as an $h$-version of integration by
parts.\\
\textbf{Fundamental theorem of \( h \)-calculus:}
	If  \( b - a \in h\mathbb{Z} \), then
	\begin{align}
			\int_a^b \Delta_{h}f(x)\, d_h x = f(b) - f(a).
	\end{align} 
\textbf{$h$-Integration by parts:}
	\begin{align}\label{hintpart}
		\int_a^b f(x) \Delta_{h}g(x)\, d_hx  = f(b)g(b) - f(a)g(a) - \int_a^b g(x+h)\Delta_{h}f(x)\, d_hx.
	\end{align} 

\section{Generalized divided differences}\label{section3}
B-splines from non-polynomial spline spaces, such as trigonometric, hyperbolic, or Chebychevian 
spaces have been studied by many authors, see e.g. the papers of Schumaker \cite{Schumaker3,schumaker1}, Lyche and Winther \cite{Lyche1}. The approach to these various types of B-splines is very similar in all cases, that is, the use of a suitable divided-difference operator, applied to an appropriate generalization of a truncated-power-function.

For functions $\varphi_{0}, \varphi_{1},\ldots,\varphi_{m} $ and  numbers $x_{0}, x_{1}, \ldots, x_{m}$, let
\begin{equation}
	\det\begin{pmatrix}
		\varphi_{0} & \varphi_{1}  &\ldots  & \varphi_{m} \\
		x_{0} & x_{1} & \ldots & x_{m} 
	\end{pmatrix}= \begin{vmatrix}
		\varphi_{0}(x_{0}) & \varphi_{1}(x_{0})  &\ldots  & \varphi_{m}(x_{0}) \\ 
		\varphi_{0}(x_{1}) & \varphi_{1}(x_{1})  &\ldots  & \varphi_{m}(x_{1}) \\
		\vdots & \vdots & & \vdots \\ 
		\varphi_{0}(x_{m}) & \varphi_{1}(x_{m})  &\ldots  & \varphi_{m}(x_{m}) 
	\end{vmatrix}.
\end{equation}
Recently,  Z\"{u}rnac{\i} and Di\c{s}ib\"{u}y\"{u}k \cite{fatma} defined non-polynomial B-spline functions  for the space 
	\begin{equation}
	\pi_{n}(\gamma_{1},\gamma_{2})=\text{span}\big\{\gamma_{1}^{n-k}\gamma_{2}^{k}\big\}_{k=0}^{n} ,
\end{equation}
where the functions $ \gamma_{1} $ and $ \gamma_{2} $ are linearly independent and 
\begin{align}\label{dfunc}
	d(x_{1},x_{2})=\gamma_{1}(x_{1})\gamma_{2}(x_{2})-\gamma_{1}(x_{2})\gamma_{2}(x_{1})\neq 0 \quad \mbox{for}\quad x_{1},x_{2} \in [a,b], \, x_{1} \neq x_{2}.
\end{align}
These non-polynomial B-splines are constructed by using  non-polynomial divided differences applied to a proper generalization of the truncated-power function. These non-polynomial divided differences are defined recursively as a generalization of the classical divided differences.
\begin{definition}
	Given $m+1$ distinct nodes $x_{0},x_{1},\ldots,x_{m}$, the non-polynomial divided difference of order $ m $ is defined recursively by setting
	\begin{align}
		&[x_{k}]_{\gamma_{1}, \gamma_{2}}f=f(x_{k}) \quad \mbox{for}\quad k=0, 1,\ldots, m,\\
		\label{divdif}
		&[x_{0},\ldots,x_{m}]_{\gamma_{1}, \gamma_{2}}f=\frac{[x_{1},\ldots,x_{m}]_{\gamma_{1}, \gamma_{2}}f - [x_{0},\ldots,x_{m-1}]_{\gamma_{1}, \gamma_{2}}f}{d(x_{0},x_{m})}.
	\end{align}
\end{definition}
These non-polynomial divided differences satisfy numerous identities and properties similar to the identities and properties of the classical polynomial divided
difference. In particular  \cite{fatma,fatma2,fatma3}:\\\\
\textit{Representation as the ratio of two determinants:}	If $1\in \text{span}\big\{\gamma_{1}, \gamma_{2}\big\}$, then\\
\begin{align*}
	[x_{0},\ldots,x_{m}]_{\gamma_{1}, \gamma_{2}}f=c_{1}^{m} \frac{	\det\begin{pmatrix}
			1   &  \gamma_{2} &\ldots  & \gamma_{2}^{m-1} &f\\
			x_{0} & x_{1} & \ldots & x_{m-1} & x_{m} 
	\end{pmatrix}}{	\det\begin{pmatrix}
			1   &  \gamma_{2}   &\ldots  &  \gamma_{2}^{m}\\
			x_{0} & x_{1} & \ldots   & x_{m} 
	\end{pmatrix}},
\end{align*}
where $c_{1}\gamma_{1}(x)+c_{2}\gamma_{2}(x)=1$ and $c_1\neq0$. \\\\
\textit{Leibniz Rule:}
\begin{equation}\label{NPleibniz}
	[x_{0},\ldots,x_{m}]_{\gamma_{1}, \gamma_{2}}(fg)=\sum_{k=0}^{m}[x_{0},\ldots,x_{k}]_{\gamma_{1}, \gamma_{2}}f \,  [x_{k},\ldots,x_{m}]_{\gamma_{1}, \gamma_{2}}g.
\end{equation}
\textit{Lagrange coefficient:} If $1\in \text{span}\big\{\gamma_{1}, \gamma_{2}\big\}$, then 
\begin{equation}\label{NPlagrange}
	[x_{0},\ldots,x_{m}]_{\gamma_{1}, \gamma_{2}}f=\sum_{k=0}^{m}\frac{f(x_{k})}{\prod \limits_{j\neq k} d(x_{j},x_{k})}.
\end{equation}
\textit{Symmetry:} $[x_{0},\ldots,x_{m}]_{\gamma_{1}, \gamma_{2}}f$
is a symmetric function  of the nodes  $x_{0},x_{1},\ldots,x_{m}$.

\section{$h$-Trigonometric divided differences}\label{section4}
For any integer $m \geq 1$, set
\begin{align}\label{space1}
S_m=&\, \mbox{span}\left\{e^{-i\left(\frac{m-1}{2}\right)x}_{h}, e^{-i\left(\frac{m-1}{2}-1\right)x}_{h}, \ldots, e^{i\left(\frac{m-1}{2}\right)x}_{h}\right\},\\ \label{expspace2}
	\tilde{S}_m=&\, \mbox{span}\left\{1, e^{ix}_{h}, \ldots, e^{i(m-1)x}_{h}\right\}.
\end{align}
Let $U_m$ and $\tilde{U}_m$ be the multiplication operators defined by
\begin{align}\label{operator11}
	\left(U_mf\right)(x)=e^{i\left(\frac{m-1}{2}\right)x}_{h}f(x),\\\label{operator2}
	\left(\tilde{U}_mf\right)(x)=e^{-i\left(\frac{m-1}{2}\right)x}_{h}f(x).
\end{align}
Then $U_m(S_m)=\tilde{S}_m$ and $\tilde{U}_m(\tilde{S}_m)=S_m$.

Let $x_{0}< x_{1}< \cdots < x_{m}$ be such that $x_{m}-x_{0}< \frac{2\pi h}{\ln(1+h)}$. For $f \in C([x_0, x_m])$,  we define the $h$-trigonometric divided differences by setting
\begin{align}\label{htrigodd}
	[x_0,\ldots, x_m;h]_{t}f= 2^{m-1}\frac{	\det\begin{pmatrix}
			1   & \cos_{h}x  &\sin_{h}x &\ldots  & \cos_{h}nx  &\sin_{h}nx & f(x)\\
			x_{0} & x_{1}  & x_{2}& \ldots & x_{m-2}& x_{m-1} & x_{m} 
	\end{pmatrix}}{	\det\begin{pmatrix}
			\cos_{h}\frac{x }{2}  &  \sin_{h}\frac{x }{2}  &\ldots  & \cos_{h}\left(n+\frac{1}{2} \right)x &  \sin_{h}\left(n+\frac{1}{2} \right)x  \\
			x_{0} & x_{1} & \ldots  & x_{m-1} & x_{m} 
	\end{pmatrix}}
\end{align}
for $m=2n+1$, and 
\begin{align}\label{htrigeven}
	[x_0,\ldots, x_m;h]_{t}f=  2^{m}\frac{	\det\begin{pmatrix}
			\cos_{h}\frac{x }{2}  &  \sin_{h}\frac{x }{2}  &\ldots  & \cos_{h}\left(n-\frac{1}{2} \right)x &  \sin_{h}\left(n-\frac{1}{2} \right)x &f(x) \\
			x_{0} & x_{1} & \ldots & x_{m-2} & x_{m-1} & x_{m} 
	\end{pmatrix}}{	\det\begin{pmatrix}
			1   & \cos_{h}x  &\sin_{h}x &\ldots  & \cos_{h}nx  &\sin_{h}nx\\
			x_{0} & x_{1}  & x_{2}& \ldots& x_{m-1} & x_{m} 
	\end{pmatrix}}
\end{align}
for $m=2n$.

In order to study  $h$-trigonometric divided differences, we also introduce an exponential version of the divided difference. 	The space $\tilde{S}_m$ is the special case of the space 	$\pi_{n}(\gamma_{1},\gamma_{2})$ for  $\gamma_{1}(x)=1$ and $\gamma_{2}(x)= e^{ix}_{h}$.  We define the $h$-exponential divided difference by setting
	\begin{equation}\label{GDDrecurrenceh}
	[ x_{0}, \ldots, x_{m} ;h ]_{e} f=	[ x_{0}, \ldots, x_{m} ]_{\gamma_{1},\gamma_{2}} f \quad \mbox{where}\quad \gamma_{1}(x)=1 \quad \mbox{and}\quad \gamma_{2}(x)= e^{ix}_{h}.
\end{equation}
The versions of the properties from Section \ref{section3} for  $h$-exponential divided differences are given below. \\\\
\textit{Representation as the ratio of two determinants:}
\begin{equation}\label{hDDdefinition}
	[ x_{0}, \ldots, x_{m} ;h ]_{e} f=\frac{	\det\begin{pmatrix}
			1   & e^{ix}_{h}  &\ldots  & e^{i(m-1)x}_{h}&f\\
			x_{0} & x_{1} & \ldots & x_{m-1} & x_{m} 
	\end{pmatrix}}{	\det\begin{pmatrix}
			1   & e^{ix}_{h}  &\ldots  & e^{imx}_{h}\\
			x_{0} & x_{1} & \ldots   & x_{m} 
	\end{pmatrix}}.
\end{equation}
Notice that
\begin{equation*}
	\det\begin{pmatrix}
		1   & e^{ix}_{h}  &\ldots  & e^{imx}_{h}\\
		x_{0} & x_{1} & \ldots   & x_{m} 
	\end{pmatrix}= \begin{vmatrix}
		1 & e^{ix_{0}}_{h}
		&\ldots  & e^{imx_{0}}_{h}  \\ 
		1& e^{ix_{1}}_{h} &\ldots  & e^{imx_{1}}_{h} \\
		\vdots & \vdots & & \vdots \\ 
		1 & e^{ix_{m}}_{h} &\ldots  & e^{imx_{m}}_{h} 
	\end{vmatrix}.
\end{equation*}
This determinant is an analogue of the Vandermonde determinant since $e^{ikx}_{h}=(e^{ix}_{h})^{k}$. Therefore
\begin{equation}\label{vandermonde1}
	\det\begin{pmatrix}
		1   & e^{ix}_{h}  &\ldots  & e^{imx}_{h}\\
		x_{0} & x_{1} & \ldots   & x_{m} 
	\end{pmatrix}= \prod_{0 \leq j < k \leq m}(e^{ix_{k}}_{h}-e^{ix_{j}}_{h}).
\end{equation}
By (\ref{identitysinexp}), 
\begin{equation}\label{vandermonde}
	\det\begin{pmatrix}
		1   & e^{ix}_{h}  &\ldots  & e^{imx}_{h}\\
		x_{0} & x_{1} & \ldots   & x_{m} 
	\end{pmatrix}=  (2i)^{m(m+1)/2}\prod_{0 \leq j < k \leq m} e^{i(x_{k}+x_{j})/2}_{h}\sin_{h}\left(\frac{x_{k}-x_{j}}{2}\right).
\end{equation}
Since the zeros of $\sin_{h}x$ are located at $x=\frac{n\pi h}{\ln(1+h)}$, $n \in \mathbb{Z}$, this determinant is different from zero for  distinct nodes and for $x_{m}-x_{0}< \frac{2\pi h}{\ln(1+h)}$.  \\\\
\textit{Recursion:}\\
\begin{align}\label{recursion}
	[ x_{0}, \ldots, x_{m} ;h ]_{e} f=\frac{	[ x_{1}, \ldots, x_{m} ;h ]_{e} f-	[ x_{0}, \ldots, x_{m-1} ;h ]_{e} f}{e^{ix_{m}}_{h}-e^{ix_{0}}_{h}	}.
\end{align}
\textit{Leibniz Rule:}
\begin{equation}\label{Leibniz}
	[ x_{0}, \ldots, x_{m} ;h ]_{e} \left(fg\right)=\sum_{k=0}^{m}[ x_{0}, \ldots, x_{k} ;h ]_{e} f[ x_{k}, \ldots, x_{m} ;h ]_{e} g.
\end{equation}	
\textit{Lagrange coefficient:}  \begin{equation}\label{GDDlagrange}
	[ x_{0}, \ldots, x_{m} ;h ]_{e} f=\sum_{j=0}^{m}\frac{f(x_{j})}{\prod \limits_{k \neq j}\left(e^{ix_{j}}_{h}-e^{ix_{k}}_{h}\right)}.
\end{equation}	\\
\textit{Symmetry:} $[x_{0},\ldots,x_{m}]_{e}f$
is a symmetric function  of the nodes  $x_{0},x_{1},\ldots,x_{m}$. \\\\
Notice that if $f\in \tilde{S}_{m}$, then $f(x)=\sum_{k=0}^{m-1}a_{k}e^{ikx}_{h}$, and (\ref{hDDdefinition})  yields $[ x_{0}, \ldots, x_{m} ;h ]_{e} f=0.$
Moreover,  given the values of a function $f$ at $m+1$ distinct points $x_0, x_1, \ldots,x_m$, there exists a unique function $P(x)=\sum_{k=0}^{m}a_{k}e^{ikx}_{h}\in \tilde{S}_{m+1}$ such that $P(x_j)=f(x_j)$, $j=0,1,\ldots,m$, since the determinant of the collocation matrix is non-zero. Moreover 
\begin{equation}\label{dcoeficient}
	[ x_{0}, \ldots, x_{m} ;h ]_{e} f=a_m.
\end{equation}
 Indeed since  $P(x_j)=f(x_j)$, it follows from (\ref{hDDdefinition}) that
\begin{equation*}
	[ x_{0}, \ldots, x_{m} ;h ]_{e} f=\frac{	\det\begin{pmatrix}
			1   & e^{ix}_{h}  &\ldots  & e^{i(m-1)x}_{h}& P\\
			x_{0} & x_{1} & \ldots & x_{m-1} & x_{m} 
	\end{pmatrix}}{	\det\begin{pmatrix}
			1   & e^{ix}_{h}  &\ldots  & e^{imx}_{h}\\
			x_{0} & x_{1} & \ldots   & x_{m} 
	\end{pmatrix}}=\frac{	\det\begin{pmatrix}
	1   & e^{ix}_{h}  &\ldots  & e^{i(m-1)x}_{h}&  a_{m}e^{imx}_{h} \\
	x_{0} & x_{1} & \ldots & x_{m-1} & x_{m} 
\end{pmatrix}}{	\det\begin{pmatrix}
1   & e^{ix}_{h}  &\ldots  & e^{imx}_{h}\\
x_{0} & x_{1} & \ldots   & x_{m} 
\end{pmatrix}}= a_{m}.
\end{equation*}

The following lemma describes the relation between the $h$-trigonometric and the 
$h$-exponential divided differences.
\begin{lemma}\label{lemmarelation}
	Suppose that $x_{0}< x_{1} < \cdots < x_{m} < x_0+ \frac{2\pi h}{\ln(1+h)}$. Then
	\begin{align}\label{relation_te}
		[x_0,\ldots, x_m;h]_{t}f=c_{0,m}[ x_{0}, \ldots, x_{m} ;h ]_{e} \left(U_mf\right),
	\end{align}
	where 
	\begin{align}\label{coefficient}
		c_{0,m}=(2i)^{m}e^{\frac{i}{2}\sum_{k=0}^{m}x_{k}}_{h}.
	\end{align}
\end{lemma}
\begin{proof}
 Recall  definition (\ref{htrigodd})-(\ref{htrigeven}) of $[x_0,\ldots, x_m;h]_{t}f$ and formulas (\ref{cosseri}) for $\sin_{h}x$ and $\cos_{h}x$. Applying elementary column operations to pairs of adjacent columns and invoking (\ref{hDDdefinition}), we obtain 
	\begin{align*}
		[x_0,\ldots, x_m;h]_{t}f&= (2i)^m\frac{	\det\begin{pmatrix}
				e^{-i\left(\frac{m-1}{2}\right)x}_{h}  &  e^{-i\left(\frac{m-3}{2}\right)x}_{h}   &\ldots  & e^{i\left(\frac{m-1}{2}\right)x}_{h}   &f(x) \\
				x_{0} & x_{1} & \ldots & x_{m-1} & x_{m} 
		\end{pmatrix}}{	\det\begin{pmatrix}
				e^{-i\left(\frac{m}{2}\right)x}_{h}   & e^{-i\left(\frac{m-2}{2}\right)x}_{h}  &\ldots   &e^{i\left(\frac{m}{2}\right)x}_{h}\\
				x_{0} & x_{1}  & \ldots & x_{m} 
		\end{pmatrix}}\\&= (2i)^m \frac{e^{-i\left(\frac{m-1}{2}\right)(x_0+x_1+ \cdots+x_m)}_{h}}{e^{-i\left(\frac{m}{2}\right)(x_0+x_1+ \cdots+x_m)}_{h}}\,\frac{	\det\begin{pmatrix}
				1   & e^{ix}_{h}  &\ldots  & e^{i(m-1)x}_{h}& e^{i\left(\frac{m-1}{2}\right)x}_{h}f(x)\\
				x_{0} & x_{1} & \ldots & x_{m-1} & x_{m} 
		\end{pmatrix}}{	\det\begin{pmatrix}
				1   & e^{ix}_{h}  &\ldots  & e^{imx}_{h}\\
				x_{0} & x_{1} & \ldots   & x_{m} 
		\end{pmatrix}}\\
		&=(2i)^{m} e^{\frac{i}{2}\sum_{k=0}^{m}x_{k}}_{h}[ x_{0}, \ldots, x_{m} ;h ]_{e}\left( e^{i\left(\frac{m-1}{2}\right)x}_{h}f(x)\right).
\end{align*} \end{proof}
\begin{lemma}
	Suppose that $x_{0}< x_{1}< \cdots < x_{m} < x_0+ \frac{2\pi h}{\ln(1+h)}$. Then for any $f \in C([x_0, x_m])$,
	\begin{align*}
		[x_{0}, \ldots, x_{m} ;h ]_{t} f=\sum_{j=0}^{m}\frac{f(x_{j})}{\prod \limits_{k \neq j}\sin_{h}\left(\left(x_{j}-x_{k}\right)/2\right)}.
	\end{align*}
\end{lemma}
\begin{proof}
	By	(\ref{GDDlagrange}), we have
	\begin{align*}
		[ x_{0},\ldots, x_{m} ;h ]_{e} \left(U_mf\right)=\sum_{j=0}^{m}\frac{e^{i\left(\frac{m-1}{2}\right)x_j}_{h}f(x_{j})}{\prod \limits_{k \neq j}\left(e^{ix_{j}}_{h}-e^{ix_{k}}_{h}\right)}.
	\end{align*}
Then Lemmma \ref{lemmarelation} and (\ref{identitysinexp}) yield
	\begin{align*}
		[x_0,\ldots, x_m;h]_{t}f	=	c_{0,m}\sum_{j=0}^{m}\frac{e^{i\left(\frac{m-1}{2}\right)x_j}_{h}f(x_{j})}{\prod \limits_{k \neq j}\left(e^{ix_{j}}_{h}-e^{ix_{k}}_{h}\right)}
		=& c_{0,m}\sum_{j=0}^{m}\frac{e^{i\left(\frac{m-1}{2}\right)x_j}_{h}f(x_{j})}{(2i)^{m}e^{\frac{i}{2}\left(mx_j+\sum_{k\neq j}^{m}x_{k}\right)}_{h}\prod \limits_{k \neq j}\sin_{h}\left(\left(x_{j}-x_{k}\right)/2\right)}\\
		=	&\sum_{j=0}^{m}\frac{f(x_{j})}{\prod \limits_{k \neq j}\sin_{h}\left(\left(x_{j}-x_{k}\right)/2\right)},
	\end{align*}
since by (\ref{coefficient}) all the other factors cancel.
\end{proof}

Next, we shall give a three-term recurrence relation for the $h$-trigonometric divided 
differences.
\begin{lemma}
Suppose that  $m \geq 2$ and  that $x_{0}< x_{1}< \cdots < x_{m} < x_0+ \frac{2\pi h}{\ln(1+h)}$. Then
	\begin{align*}
		[x_0,\ldots, x_m;h]_{t}f=&\gamma_{m}	[x_2,\ldots, x_m;h]_{t}f+ \beta_{m}	[x_1,\ldots, x_{m-1};h]_{t}f\\ &+\alpha_{m}	[x_0,\ldots, x_{m-2};h]_{t}f,
	\end{align*} 
	where
	\begin{align}
		\nonumber	&\gamma_{m}= \left[\left(\sin_{h}\frac{x_{m}-x_{0}}{2}\right)\left(\sin_{h}\frac{x_{m}-x_{1}}{2}\right)\right]^{-1},\\
		\nonumber	&\beta_{m}=\scalebox{1.12}{$-\sin_{h}\left(\frac{x_{m}+x_{m-1}-x_{1}-x_{0}}{2}\right) \left[\left(\sin_{h}\frac{x_{m}-x_{0}}{2}\right)\left(\sin_{h}\frac{x_{m}-x_{1}}{2}\right)\left(\sin_{h}\frac{x_{m-1}-x_{0}}{2}\right)\right]^{-1}$},\\ \nonumber
		&\alpha_{m}= \left[\left(\sin_{h}\frac{x_{m}-x_{0}}{2}\right)\left(\sin_{h}\frac{x_{m-1}-x_{0}}{2}\right)\right]^{-1}.
	\end{align}
\end{lemma}
\begin{proof}
	Let
	\begin{align*}
		c_{j,m}=(2i)^{m}e^{\frac{i}{2}\sum_{k=0}^{m}x_{k+j}}_h.
	\end{align*}
	By (\ref{relation_te}) of Lemma \ref{lemmarelation},
	\begin{align*}
		[x_0,\ldots, x_m;h]_{t}f=c_{0,m}[ x_{0},  \ldots, x_{m} ;h ]_{e}\left( U_mf\right) .
	\end{align*}
	Using (\ref{recursion}), 
	\begin{align*}
		(e^{ix_m}_{h}-e^{ix_0}_{h})[ x_{0}, \ldots, x_{m} ;h ]_{e} \left( U_mf\right) =[  x_{1}, \ldots, x_{m} ;h ]_{e} \left( U_mf\right) -[ x_{0},  \ldots, x_{m-1} ;h ]_{e} \left( U_mf\right) .
	\end{align*}
	Since $(U_mf)(x)=e^{ix}_{h}(U_{m-2}f)(x)$, we may write
	\begin{align*}
		(e^{ix_m}_{h}-e^{ix_0}_{h})[ x_{0}, \ldots, x_{m} ;h ]_{e} \left( U_mf\right) =[  x_{1}, \ldots, x_{m} ;h ]_{e} \left( e^{ix}_{h}U_{m-2}f\right) -[ x_{0},  \ldots, x_{m-1} ;h ]_{e} \left(e^{ix}_{h}U_{m-2}f \right).
	\end{align*}
	By  (\ref{Leibniz}), 
	\begin{align*}
		(e^{ix_m}_{h}-e^{ix_0}_{h})[ x_{0}, \ldots, x_{m} ;h ]_{e} \left(U_mf\right)=& e^{ix_{1}}_{h}[  x_{1}, \ldots, x_{m} ;h ]_{e}\left( U_{m-2}f \right) +[  x_{2}, \ldots, x_{m} ;h ]_{e} \left( U_{m-2}f\right) \\-&e^{ix_{m-1}}_{h}[ x_{0},  \ldots, x_{m-1} ;h ]_{e} \left(U_{m-2}f \right)\\-&[ x_{0},  \ldots, x_{m-2} ;h ]_{e} \left( U_{m-2}f\right) .
	\end{align*}
	Applying (\ref{recursion}) to the divided differences of order $m - 1$ and then combining pairs of like terms yields
	\begin{align*}
		&(e^{ix_m}_{h}-e^{ix_0}_{h})[ x_{0}, \ldots, x_{m} ;h ]_{e} \left( U_mf\right)\\&=\frac{e^{ix_{1}}_{h}}{e^{ix_{m}}_{h}-e^{ix_{1}}_{h}} \left\{ [  x_{2}, \ldots, x_{m} ;h ]_{e} \left(U_{m-2}f \right)-[  x_{1}, \ldots, x_{m-1} ;h ]_{e}  \left(U_{m-2}f \right)\right\}\\&+[  x_{2}, \ldots, x_{m} ;h ]_{e}  \left(U_{m-2}f \right)\\&-\frac{e^{ix_{m-1}}_{h}}{e^{ix_{m-1}}_{h}-e^{ix_{0}}_{h}}\left\{ [ x_{1},  \ldots, x_{m-1} ;h ]_{e}  \left(U_{m-2}f \right)-[ x_{0},  \ldots, x_{m-2} ;h ]_{e}  \left(U_{m-2}f \right)\right\}\\&-[ x_{0},  \ldots, x_{m-2} ;h ]_{e}  \left(U_{m-2}f \right).
\\&=	\frac{e^{ix_{m}}_{h}}{\left( e^{ix_{m}}_{h}-e^{ix_{1}}_{h}\right)}  [  x_{2}, \ldots, x_{m} ;h ]_{e}  \left(U_{m-2}f \right)-	\frac{e^{i(x_{m}+x_{m-1})}_{h}-e^{i(x_{1}+x_{0})}_{h}}{(e^{ix_m}_{h}-e^{ix_1}_{h})\left( e^{ix_{m-1}}_{h}-e^{ix_{0}}_{h} \right)} [ x_{1},  \ldots, x_{m-1} ;h ]_{e}  \left(U_{m-2}f \right)\\&+	\frac{e^{ix_0}_{h}}{(e^{ix_{m-1}}_{h}-e^{ix_0}_{h})}[ x_{0},  \ldots, x_{m-2} ;h ]_{e}  \left(U_{m-2}f \right).
	\end{align*}
	Multipliying both sides by $c_{0,m}/(e^{ix_{m}}_{h}-e^{ix_0}_{h})$ and then applying  (\ref{relation_te})-(\ref{coefficient}) of Lemma \ref{lemmarelation} and  (\ref{identitysinexp}) completes the proof.
\end{proof}

\section{$h$-Trigonometric B-splines} \label{section5}
In preparation for the definition of the $h$-trigonometric B-splines, we start by defining the $h$-analogue of raising
a function to an integer power:
\begin{align*}
	&f(x)^{0}_{h}=1,\\
	&f(x)^{m}_{h}=\prod_{j=0}^{m-1} f(x-jh) \quad \mbox{for}\quad m\geq1.
\end{align*}
Then
\begin{align}\label{derivativepower}
\Delta_{h}\left(f(x)^{m}_{h}\right)=\left(\frac{f(x+h)-f(x-(m-1)h)}{h}\right)f(x)^{m-1}_{h} \quad \mbox{for}\quad m\geq1.
\end{align}
Moreover, we shall have occasion to use the following formulas treating $y$ as a constant and $x$ as a variable:
\begin{align}\nonumber
	&\left(e^{iy}_{h}-e^{ix}_{h}\right)^{0}_{h}=1,\\\nonumber
	&\left(e^{iy}_{h}-e^{ix}_{h}\right)^{m}_{h}= \prod_{j=0}^{m-1}\left(e^{iy}_{h}-e^{i(x-jh)}_{h}\right) \quad \mbox{for}\quad m\geq1,
\end{align} 
and 
\begin{align*}\nonumber
	&\left(\sin_{h}\frac{y-x}{2}\right)^{0}_{h}=1,\\
	&\left(\sin_{h}\frac{y-x}{2}\right)^{m}_{h}= \prod_{j=0}^{m-1} \sin_{h}\left(\frac{y-(x-jh)}{2}\right) \quad \mbox{for}\quad m\geq1.
\end{align*}
With this preparation, we are now ready to define the $h$-trigonometric B-splines.
\begin{definition}[\textbf{$h$-trigonometric B-splines}]\label{deftrigh}
	Let  $\{x_i\}$ be a non-decreasing sequence of real numbers satisfying  $0<x_{j+m}-x_{j}< \frac{2\pi h}{\ln(1+h)}$ for all $j$, where $m \geq 1$ is a fixed integer. We define the $h$-trigonometric B-splines $\left\{ T_{j,m}(x;h)\right\}$ by setting  $T_{j,m}(x;h)\equiv0$ if $x_{j+m}=x_j$, and if $x_{j+m}>x_{j}$
\begin{align}\nonumber
	T_{j,1}(x;h)=\left\{\begin{array}{ll}		1\Big/ \sin_{h}\left(\frac{x_{j+1}-x_j}{2}\right),	& \mbox{if}\quad x_j \leq x < x_{j+1}\\
		0,	& \mbox{otherwise},
	\end{array}\right.\\
\label{htrigbspline}
	T_{j,m}(x;h)=[x_{j},\ldots, x_{j+m};h]_{t}\left(\sin_{h}\frac{y-x}{2}\right)^{m-1}_{+} \quad \mbox{for} \quad m\geq 1,
\end{align}
where
\begin{eqnarray}\nonumber
	\left(\sin_{h}\frac{y-x}{2}\right)^{m-1}_{+}=\left\{\begin{array}{ll}
		\left(\sin_{h}	\frac{y-x}{2}\right)^{m-1}_{h},	& \mbox{if}\quad y>x,\\
		0,	& \mbox{otherwise},
	\end{array}\right.
\end{eqnarray}is  the $h$-trigonometric version of the truncated-power function. The divided difference in  (\ref{htrigbspline}) is taken with respect to $y$, and $0^0$ is defined to be $0$.
\end{definition}
Similar to (\ref{htrigbspline}), we define a family of complex-valued functions $E_{j,m}(x;h)$  by setting $E_{j,m}(x;h)\equiv0$ if $x_{j+m}=x_j$, and if $x_{j+m}>x_{j}$
\begin{align}\nonumber
	E_{j,1}(x;h)=\left\{\begin{array}{ll}
		1\Big/ \left(e^{ix_{j+1}}_{h}-e^{ix_j}_{h} \right),	& \mbox{if}\quad x_j \leq x < x_{j+1},\\
		0,	& \mbox{otherwise},
	\end{array}\right.\\
\label{hexpbspline}
	E_{j,m}(x;h)=   [ x_{j}, \ldots, x_{j+m}; h ]_{e} \left(e^{iy}_{h}-e^{ix}_{h}\right)^{m-1}_{+},
\end{align}
where
\begin{eqnarray}\nonumber
	\left(e^{iy}_{h}-e^{ix}_{h}\right)^{m-1}_{+}=\left\{\begin{array}{ll}
		\left(e^{iy}_{h}-e^{ix}_{h}\right)^{m-1}_{h}	& \mbox{if}\quad y>x,\\
		0,	& \mbox{otherwise},
	\end{array}\right.
\end{eqnarray} is the $h$-exponential version of the truncated-power function. Once again, the divided difference in (\ref{hexpbspline}) is taken with respect to $y$, and $0^0$ is defined to be $0$.
\begin{lemma}
	If $x_{j+m}>x_{j}$ then $	E_{j,m}(x;h)$ satisfies the recurrence relation
	\begin{align}\label{recurrencebsplinee}
		E_{j,m}(x;h)=\frac{e^{i(x-(m-2)h)}_{h}-e^{ix_{j}}_{h}}{e^{ix_{j+m}}_{h}-e^{ix_{j}}_{h}} 	E_{j,m-1}(x;h)+\frac{e^{ix_{j+m}}_{h}-e^{i(x-(m-2)h)}_{h}}{e^{ix_{j+m}}_{h}-e^{ix_{j}}_{h}} 	E_{j+1,m-1}(x;h). 
	\end{align}
\end{lemma}
\begin{proof} We apply the Leibniz rule (\ref{Leibniz}) to the $m-$th order divided difference of the function in (\ref{hexpbspline}),
	\begin{align*}
		\left(e^{iy}_{h}-e^{ix}_{h}\right)^{m-1}_{+}=\left(e^{iy}_{h}-e^{ix}_{h}\right)^{m-2}_{+}\left(e^{iy}_{h}-e^{i(x-(m-2)h)}_{h}\right).
	\end{align*}
 Then 
	\begin{align*}
		E_{j,m}(x;h)&= [ x_{j},\ldots, x_{j+m}; h ]_{e} \left(e^{iy}_{h}-e^{ix}_{h}\right)^{m-2}_{+}\left(e^{ix_{j+m}}_{h}-e^{i(x-(m-2)h)}_{h}\right)\\&+
		[ x_{j}, \ldots, x_{j+m-1}; h ]_{e} \left(e^{iy}_{h}-e^{ix}_{h}\right)^{m-2}_{+}
	\end{align*}
	since all $h$-exponential divided differences in $y$ of $e^{iy}_{h}-e^{ix-(m-2)h}_{h}$ of order $2$ and higher vanish. Then, the recurrence relation (\ref{recursion}) yields
	\begin{align*}
		&E_{j,m}(x;h)\\&=\frac{e^{ix_{m+j}}_{h}-e^{i(x-(m-2)h)}_{h}}{e^{ix_{j+m}}_{h}-e^{ix_{j}}_{h}} \left( [ x_{j+1}, \ldots, x_{j+m}; h ]_{e} \left(e^{iy}_{h}-e^{ix}_{h}\right)^{m-2}_{+}-[ x_{j}, \ldots, x_{j+m-1}; h ]_{e} \left(e^{iy}_{h}-e^{ix}_{h}\right)^{m-2}_{+}\right)\\&
		+[ x_{j},\ldots, x_{j+m-1}; h ]_{e} \left(e^{iy}_{h}-e^{ix}_{h}\right)^{m-2}_{+}.
	\end{align*}
Now the result follows from  (\ref{hexpbspline}).
\end{proof}
\begin{lemma}
	\begin{align} \label{derivativeE}
		\Delta_hE_{j,m}(x;h)&= \frac{ \left(e^{-i(m-1)h}_{h}-1\right)e^{i(x+h)}_{h} }{\left(e^{ix_{j+m}}_{h}-e^{ix_{j}}_{h}\right)h} \left(E_{j+1,m-1}(x;h)-	E_{j,m-1}(x;h) \right).
	\end{align}
\end{lemma}
\begin{proof}
	Applying the $h$-derivative to $(\ref{hexpbspline})$  and using (\ref{derivativepower}), and  (\ref{recursion}) gives
	\begin{align*}
		\Delta_hE_{j,m}(x;h)=&\left(\frac{e^{-i(m-1)h}_{h}-1}{h}\right)e^{i(x+h)}_{h}[ x_{j}, \ldots, x_{j+m}; h ]_{e} \left(e^{iy}_{h}-e^{ix}_{h}\right)^{m-2}_{+}\\ 
	=& \frac{ \left(e^{-i(m-1)h}_{h}-1\right)e^{i(x+h)}_{h} }{\left(e^{ix_{j+m}}_{h}-e^{ix_{j}}_{h}\right)h}    \Bigg([ x_{j+1}, \ldots, x_{j+m}; h ]_{e} \left(e^{iy}_{h}-e^{ix}_{h}\right)^{m-2}_{+}\\&-[ x_{j}, \ldots, x_{j+m-1}; h ]_{e} \left(e^{iy}_{h}-e^{ix}_{h}\right)^{m-2}_{+}\Bigg).
	\end{align*} 
 Then (\ref{derivativeE}) follows from this equation and (\ref{hexpbspline}).
\end{proof}
\begin{proposition}\label{relationbsplines}
	Let $\tilde{U}_m$ be the multiplication operator defined in (\ref{operator2}). Then,
	\begin{eqnarray}\label{eq.relationbsplines}
		T_{j,m}(x;h)=2ie^{\frac{i}{2}\left(\sum_{k=0}^{m}x_{j+k}+ \binom{m-1}{2}h\right)}_{h}\tilde{U}_{m}E_{j,m}(x;h).
	\end{eqnarray}
\end{proposition}
\begin{proof}
	By Definition \ref{deftrigh} and Lemma \ref{lemmarelation}, we can rewrite $T_{j,m}(x;h)$ as
	\begin{equation*}
		T_{j,m}(x;h)=(2i)^{m}e^{\frac{i}{2}\sum_{k=0}^{m}x_{j+k}}_{h}[ x_{j}, \ldots, x_{j+m} ;h ]_{e}\left\{ e^{i\left(\frac{m-1}{2}\right)y}_{h}\left(\sin_{h}\frac{y-x}{2}\right)^{m-1}_{+} \right\} .
	\end{equation*}
	Applying (\ref{identitysinexp}) to $\left(\sin_{h}\frac{y-x}{2}\right)^{m-1}_{+}$ yields
	\begin{align*}
		\prod_{j=0}^{m-2}\sin_{h}\left(\frac{y-(x-jh)}{2}\right)=\left(2i\right)^{1-m}e^{-\frac{i}{2}(m-1)\left(x+y\right)+\frac{i}{2}\binom{m-1}{2}h}_{h}\prod_{j=0}^{m-2}\left(e^{iy}_{h}-e^{i(x-jh)}_{h}\right).
	\end{align*}
	Therefore
	\begin{align*}
		T_{j,m}(x;h)=2ie^{\frac{i}{2}\sum_{k=0}^{m}x_{j+k}}_{h}[ x_{j}, \ldots, x_{j+m} ;h ]_{e} \left\{ e^{-i(\frac{m-1}{2})x+\frac{i}{2}\binom{m-1}{2}h}_{h}     	 \left(e^{iy}_{h}-e^{ix}_{h}\right)^{m-1}_{+}\right\}.
	\end{align*}
Now we get (\ref{eq.relationbsplines}) by (\ref{operator2}) and (\ref{hexpbspline}).
\end{proof}
\begin{theorem}[\textbf{Recurrence relation for $h$-trigonometric B-splines}]
	Suppose that $m \geq 2$ and $0<  x_{j+m}-x_{j} <  \frac{2\pi h}{\ln(1+h)}$. Then  $T_{j,m}(x;h)$ satisfies the recurrence relation
	\begin{align}\label{htrigbsplineRR}
		T_{j,m}(x;h)=\frac{\sin_{h}\frac{x-(m-2)h-x_{j}}{2}}{\sin_{h}\frac{x_{j+m}-x_{j}}{2}}\, 	T_{j,m-1}(x;h)+\frac{\sin_{h}\frac{x_{j+m}-x+(m-2)h}{2}}{\sin_{h}\frac{x_{j+m}-x_{j}}{2}}\,	T_{j+1,m-1}(x;h). 
	\end{align}
\end{theorem}
\begin{proof}
	By Proposition \ref{relationbsplines} and the recurrence relation (\ref{recurrencebsplinee}) for $E_{j,m}(x;h)$, we can write $T_{j,m}(x;h)$ as
	\begin{align*}
		&T_{j,m}(x;h)\\&=2ie^{\frac{i}{2}\left(\sum_{k=0}^{m}x_{k+j}-(m-1)x+\binom{m-1}{2}h\right)}_{h} \Bigg\{\frac{e^{i(x-(m-2)h)}_{h}-e^{ix_{j}}_{h}}{e^{ix_{j+m}}_{h}-e^{ix_{j}}_{h}} 	E_{j,m-1}(x;h)\\&+\frac{e^{ix_{j+m}}_{h}-e^{i(x-(m-2)h)}_{h}}{e^{ix_{j+m}}_{h}-e^{ix_{j}}_{h}} 	E_{j+1,m-1}(x;h) \Bigg\}.
	\end{align*}		
	Applying (\ref{identitysinexp}), we get
	\begin{align*}
		&T_{j,m}(x;h)\\
		&=2i\frac{e^{\frac{i}{2}\left(\sum_{k=1}^{m-1}x_{k+j}+\binom{m-1}{2}h-(m-1)x\right)}_{h} }{\sin_{h}\left(\frac{x_{j+m}-x_{j}}{2}\right)}\Bigg\{e^{\frac{i}{2}(x_{j}+x-(m-2)h)}_{h}\sin_{h}\left(\frac{x-(m-2)h-x_{j}}{2}\right)	E_{j,m-1}(x;h)\\&+ e^{\frac{i}{2}(x_{j+m}+x-(m-2)h)}_{h}\sin_{h}\left(\frac{x_{j+m}-x+(m-2)h}{2}\right)	E_{j+1,m-1}(x;h) \Bigg\},
	\end{align*}		
which by (\ref{eq.relationbsplines}) reduces to the right-hand side of (\ref{htrigbsplineRR}).
\end{proof}
\begin{theorem}[\textbf{$h$-derivative formula for $h$-trigonometric B-splines}]
	Let $m \geq 2$ and $0 < x_{j+m}- x_{j}  <  \frac{2\pi h}{\ln(1+h)}$. Then  
	\begin{align}\nonumber
		\Delta_h T_{j,m}(x;h) &=ie^{\frac{i}{4}(m-1)h }_{h}\frac{\left(e^{-i(m-1)h/2}_{h}-1\right)}{h\sin_{h}\left(\frac{x_{j+m}-x_{j}}{2}\right) }\Bigg\{\cos_{h}\left(\frac{x-(m-3)h/2-x_{j}}{2}\right)\, 	T_{j,m-1}(x;h)\\ \label{htrigderivative}&-\cos_{h}\left(\frac{x_{j+m}-x+(m-3)h/2}{2}\right)\,	T_{j+1,m-1}(x;h)\Bigg\}.
	\end{align}	
\end{theorem}
\begin{proof}
	Taking the $h$-derivative of (\ref{eq.relationbsplines}) and applying the $h$-product rule (\ref{derivpr}), we obtain
	\begin{align*}
		\Delta_h &T_{j,m}(x;h)\\=&2ie^{\frac{i}{2}\left(\sum_{k=0}^{m}x_{k+j}+ \binom{m-1}{2}h\right)}_{h}\bigg(\frac{e^{-i(m-1)h/2}_{h}-1}{h}e^{-i\frac{(m-1)}{2}x}_{h}	E_{j,m}(x;h)+e^{-\frac{i}{2}(m-1)(x+h)}_{h}	\Delta_h E_{j,m}(x;h)\bigg)\\
		=&2ie^{\frac{i}{2}\left(\sum_{k=0}^{m}x_{k+j}+ \binom{m-1}{2}h-(m-1)x\right)}_{h}\left(\frac{e^{-i(m-1)h/2}_{h}-1}{h}	E_{j,m}(x;h)+e^{-\frac{i}{2}(m-1)h}_{h}	\Delta_h E_{j,m}(x;h)\right).
	\end{align*}
	Next, by (\ref{recurrencebsplinee}) and (\ref{derivativeE}), 
	\begin{align*}
		&\Delta_h T_{j,m}(x;h)\\=&   \frac{2ie^{\frac{i}{2}\left(\sum_{k=0}^{m}x_{k+j}+ \binom{m-1}{2}h-(m-1)x\right)}_{h}}{e^{ix_{j+m}}_{h}-e^{ix_{j}}_{h}}\Biggl\{\left(\frac{e^{-i(m-1)h/2}_{h}-1}{h}\right)\Bigg(\left(     e^{i(x-(m-2)h)}_{h}-e^{ix_{j}}_{h}\right)E_{j,m-1}(x;h)\\&+\left(     e^{ix_{j+m}}_{h}-e^{i(x-(m-2)h)}_{h}\right)	E_{j+1,m-1}(x;h)\Bigg)	\\&+              
		e^{-\frac{i}{2}(m-1)h}_{h} \frac{\left(e^{-i(m-1)h}_{h}-1\right)}{h}e^{i(x+h)}_{h} \bigg(E_{j+1,m-1}(x;h) - E_{j,m-1}(x;h) \bigg)	\Biggr\}\\
	=& \frac{2ie^{\frac{i}{2}\left(\sum_{k=0}^{m}x_{k+j}+ \binom{m-1}{2}h-(m-1)x\right)}_{h}}{e^{ix_{j+m}}_{h}-e^{ix_{j}}_{h}}\left(\frac{e^{-i(m-1)h/2}_{h}-1}{h}\right)\Biggl\{ -E_{j,m-1}(x;h) \left(e^{ix_{j}}_{h}+e^{i(x-(m-3)h/2)}_{h} \right)\\&+E_{j+1,m-1}(x;h)\left(e^{ix_{j+m}}_{h}+e^{i(x-(m-3)h/2)}_{h}	\right)\Biggr\},
	\end{align*}
	after combining the pairs of like terms involving $E_{j,m-1}(x;h)$ and $E_{j+1,m-1}(x;h)$.
	Applying  (\ref{identitysinexp}) and (\ref{identitycosexp}), we obtain
	\begin{align*}
		&\Delta_h T_{j,m}(x;h)\\=&    \frac{e^{\frac{i}{2}\left(\sum_{k=0}^{m}x_{k+j}+ \binom{m-1}{2}h-(m-1)x\right)}_{h}}{e^{\frac{i}{2}(x_j+x_{j+m})}_{h}\sin_{h}(\frac{x_{j+m}-x_j}{2})}\left(\frac{e^{-i(m-1)\frac{h}{2}}_{h}-1}{h}\right)      \Biggl\{ -2E_{j,m-1}(x;h)e^{\frac{i}{2}(x_{j}+x-(m-3)\frac{h}{2})}_{h}\cos_{h}\left(\frac{x-(m-3)\frac{h}{2}-x_{j}}{2}\right) \\&+2E_{j+1,m-1}(x;h)e^{\frac{i}{2}(x_{j+m}+x-(m-3)\frac{h}{2})}_{h}\cos_{h}\left(\frac{x_{j+m}-x+(m-3)\frac{h}{2}}{2}\right)	\Biggr\} \\
		=&    \frac{e^{-i(m-1)\frac{h}{2}}_{h}-1}{h\sin_{h}(\frac{x_{j+m}-x_j}{2})}     \Biggl\{ -2E_{j,m-1}(x;h)e^{\frac{i}{2}\left(\sum_{k=0}^{m-1}x_{k+j}+\left(\binom{m-1}{2}-\frac{m-3}{2}\right)h -(m-2)x\right)}_{h}\cos_{h}\left(\frac{x-(m-3)\frac{h}{2}-x_{j}}{2}\right) \\&+2E_{j+1,m-1}(x;h)e^{\frac{i}{2}\left(\sum_{k=1}^{m}x_{k+j}+\left(\binom{m-1}{2}-\frac{m-3}{2}\right)h -(m-2)x\right)}_{h}\cos_{h}\left(\frac{x_{j+m}-x+(m-3)\frac{h}{2}}{2}\right)	\Biggr\}.                    	                         
	\end{align*}
	The last two lines reduce to the right-hand side of (\ref{htrigderivative}) after applying the relation $\binom{m-1}{2}-\frac{m-3}{2}=\binom{m-2}{2}+\frac{m-1}{2}$ and then applying (\ref{eq.relationbsplines}).
\end{proof}

Let
\begin{equation}
	\tilde{E}_{j,m}(x;h)=\left( e^{ix_{j+m}}_{h}-e^{ix_{j}}_{h}\right) E_{j,m}(x;h)
\end{equation}
and
\begin{equation}
	\tilde{T}_{j,m}(x;h)=\sin_{h}\left(\frac{x_{j+m}-x_{j}}{2}\right) T_{j,m}(x;h).
\end{equation}
By (\ref{eq.relationbsplines}) and (\ref{identitysinexp}),
\begin{align}\nonumber
	\tilde{T}_{j,m}(x;h)=&\sin_{h}\left(\frac{x_{j+m}-x_{j}}{2}\right)2ie^{\frac{i}{2}\left(\sum_{k=0}^{m}x_{j+k}+ \binom{m-1}{2}h\right)}_{h}\tilde{U}_{m}E_{j,m}(x;h)\\\nonumber
	=&\left(e^{ix_{j+m}}_{h}-e^{ix_{j}}_{h}\right)e^{\frac{i}{2}\left(\sum_{k=1}^{m-1}x_{j+k}+\binom{m-1}{2}h\right)}_{h}\tilde{U}_{m}E_{j,m}(x;h)\\\label{tbspline marsden}=&e^{\frac{i}{2}\left(\sum_{k=1}^{m-1}x_{j+k}+ \binom{m-1}{2}h\right)}_{h}\tilde{U}_{m}\tilde{E}_{j,m}(x;h).
\end{align}
Now, let $(x_k,x_r)$ be any nonempty interval such that at least $m$ of the values $x_j \leq x_k$  and at least $m$ of the values $x_j \geq x_r$. Next, we give an $h$-analog of  Marsden's  identity.
\begin{theorem}[\textbf{Marsden identity, $h$-exponential form}]For $x\in (x_k,x_r)$ and $y\in \mathbb{R}$,
	\begin{equation}\label{exp_marsden}
		\left(e^{iy}_{h}-e^{ix}_{h}\right)^{m-1}_{h}=\sum_{j}\tilde{E}_{j,m}(x;h)\prod_{l=1}^{m-1}\left(e^{iy}_{h}-e^{ix_{j+l}}_{h}\right).
	\end{equation}
\end{theorem}
\begin{proof} 
	Let $w_j(y)=\left(e^{iy}_{h}-e^{ix_{j+1}}_{h}\right)\cdots\left(e^{iy}_{h}-e^{ix_{j+m-1}}_{h}\right)$ à be the product on the right-hand side of (\ref{exp_marsden}) and let $P_j(y) \in \tilde{S}_{m}$ interpolate 	$\left(e^{iy}_{h}-e^{ix}_{h}\right)^{m-1}_{+}$ at $y=x_j,\ldots,x_{j+m-1}$. Since $P_{j+1}(y)$ interpolates 	$\left(e^{iy}_{h}-e^{ix}_{h}\right)^{m-1}_{h}$ at $y=x_{j+1},\ldots,x_{j+m}$, it follows that, for a fixed $x$, the difference $P_j(y)-P_{j+1}(y)$, as a polynomial in $e^{iy}_{h}$, is equal to zero at $y=x_{j+1},\ldots,x_{j+m-1}$. Hence,
	\begin{equation}
		P_{j+1}(y)-P_{j}(y)=c(x)w_j(y).
	\end{equation}
 The coefficient $c(x)$ is the difference of the leading coefficients of $P_{j+1}(y)$ and $P_j(y)$, which by  (\ref{dcoeficient}) and (\ref{hexpbspline}) is equal to
	\begin{align*}
		c(x)=&\left( [ x_{j+1}, \ldots, x_{j+m}; h ]_{e} -[ x_{j}, \ldots, x_{j+m-1}; h ]_{e} \right)\left(e^{iy}_{h}-e^{ix}_{h}\right)^{m-1}_{+}\\=&\left(e^{ix_{j+m}}_{h}-e^{ix_j}_{h}\right)[ x_{j}, \ldots, x_{j+m}; h ]_{e}\left(e^{iy}_{h}-e^{ix}_{h}\right)^{m-1}_{+}\\=&
		\left(e^{ix_{j+m}}_{h}-e^{ix_j}_{h}\right) E_{j,m}(x;h)\\=&
		\tilde{E}_{j,m}(x;h).
	\end{align*}
Hence,
	\begin{equation*}
		P_{j+1}(y)-P_{j}(y)= w_j(y)\tilde{E}_{j,m}(x;h).
	\end{equation*}
	Let $x\in (x_k,x_r)$.	Summing both sides of this equation over $j=k-m+1, \ldots, r-1$, we obtain
	\begin{align}\label{marsdenyardimci}
		\sum_{j}w_j(y)\tilde{E}_{j,m}(x;h)=\sum_{j}\left(		P_{j+1}(y)-P_{j}(y)\right)=P_r(y)-P_{k-m+1}(y).
	\end{align}
By construction,
  $P_{k-m+1}(y)$ interpolates 	$\left(e^{iy}_{h}-e^{ix}_{h}\right)^{m-1}_{+}$ at $y=x_{k-m+1},\ldots,x_{k}$ while  $P_{r}(y)$ interpolates\\	$\left(e^{iy}_{h}-e^{ix}_{h}\right)^{m-1}_{+}$ at $y=x_{r},\ldots,x_{r+m-1}$. 
	Since $x > x_k$, we have  $P_{k-m+1}(y)=\left(e^{iy}_{h}-e^{ix}_{h}\right)^{m-1}_{+}=0$ at $y=x_{k-m+1},\ldots,x_{k}$, so $P_{k-m+1}(y)=0$, $y\in \mathbb{R}$. Also,
	since $x< x_r$, we have $P_{r}(y)=\left(e^{iy}_{h}-e^{ix}_{h}\right)^{m-1}_{+}=\left(e^{iy}_{h}-e^{ix}_{h}\right)^{m-1}_{h}$ at $y=x_{r},\ldots,x_{r+m-1}$, so $P_{r}(y)=\left(e^{iy}_{h}-e^{ix}_{h}\right)^{m-1}_{h}$, $y\in \mathbb{R}$. Therefore
	\begin{align}\label{marsdenyardimci3}
		P_r(y)-P_{k-m+1}(y)=\left(e^{iy}_{h}-e^{ix}_{h}\right)^{m-1}_{h}, \quad y\in \mathbb{R}
	\end{align}
	and (\ref{exp_marsden}) follows by (\ref{marsdenyardimci})-(\ref{marsdenyardimci3}).
\end{proof}

\begin{theorem}[\textbf{Marsden identity, $h$-trigonometric form}]For $x\in (x_k,x_r)$ and $y\in \mathbb{R}$,
	\begin{equation}
		\left(\sin_{h}\frac{y-x}{2}\right)^{m-1}_{h}=\sum_{j}\tilde{T}_{j,m}(x;h)\prod_{l=1}^{m-1}\sin_{h}\left(\frac{y-x_{j+l}}{2}\right).
	\end{equation}
\end{theorem}
\begin{proof} This result is easily derived using (\ref{exp_marsden}), (\ref{tbspline marsden}), (\ref{operator2}) and (\ref{identitysinexp}).
\end{proof}

To conclude this section, we embark on
a proof of an integral representation of the $h$-divided difference. To proceed, we
introduce the following notation.\\
\indent

Let $L_{m}$ be the difference operator defined by $L_{0}=I$ (the identity operator) and

\begin{equation}\label{E1}
	L_{m}=\prod_{j=0}^{m-1}\left(\Delta_{h}-\frac{\left(e_{h}^{i(j-(m-1) / 2) h}-1\right)}{h} I\right), \quad m \geq1.
\end{equation}
Let $M_{m}$ be the difference operator defined by $M_{0}=I$ and
\begin{equation}\label{E2}
	M_{m}=\prod_{j=0}^{m-1}\left(\Delta_{h}-\frac{\left(e_{h}^{i j h}-1\right)}{h} I\right), \quad m \geq 1.
\end{equation}
From these definitions, (\ref{space1}), (\ref{expspace2}), and (\ref{1derivative_e2}), it is clear that $S_{m} \subseteq \operatorname{Ker}\left(L_{m}\right)$ and $\tilde{S}_{m} \subseteq \operatorname{Ker}\left(M_{m}\right)$.
\begin{lemma}
	Let $U_{m}$ and $\tilde{U}_{m}$ be the multiplication operators defined by (\ref{operator11}) and (\ref{operator2}). Then,
\begin{equation}\label{E3}
	L_{m}=e_{h}^{-i\binom{m}{2} h} \tilde{U}_{m} M_{m} U_{m}, \quad m \geqslant 0.
\end{equation}
\end{lemma} 
\begin{proof}
From (\ref{derivpr}) and (\ref{1derivative_e2}), it follows that
\begin{align}\label{E4}
	\Delta_{h}\left(e_{h}^{i c x} f(x)\right) =f(x) \Delta_{h}\left(e_{h}^{i c x}\right)+e_{h}^{i c(x+h)} \Delta_{h} f(x) 
	 =e_{h}^{i c x}\left(\frac{\left(e_{h}^{i c h}-1\right)}{h} f(x)+e_{h}^{i c h} \Delta_{h} f(x)\right).
\end{align}
Let $\lambda_{c}$ denote the multiplication operator $\lambda_{c} f(x)=e_{h}^{i c x} f(x)$. Then $\lambda_{c}^{-1}=\lambda_{-c}$ and the operator relation (\ref{E4}) becomes
\begin{equation}\label{E5}
	\lambda_{c}^{-1} \Delta_{h} \lambda_{c}=e_{h}^{i c h} \Delta_{h}+\frac{\left(e_{h}^{i c h}-1\right)}{h} I. 
\end{equation}
Note also that $U_{m}=\lambda_{(m-1) / 2}$ and $\tilde{U}_{m}=\lambda_{-(m-1) / 2}$. Let $M=\prod \limits_{j=0}^{N}\left(\Delta_{h}+\alpha_{j} I\right)$. By (\ref{E5}) we have
\begin{align}\nonumber
	\lambda_{c}^{-1}M \lambda_{c} & =\prod_{j=0}^{N}\left(\lambda_{c}^{-1}\left(\Delta_{h}+\alpha_{j} I\right) \lambda_{c}\right)=\prod_{j=0}^{N}\left(e_{h}^{i c h} \Delta_{h}+\frac{\left(e_{h}^{i c h}-1+h \alpha_{j}\right)}{h} I\right) \\ \label{E6}
	& =e_{h}^{i(N+1) c h} \prod_{j=0}^{N}\left(\Delta_{h}+\frac{\left(1+\left(h \alpha_{j}-1\right) e_{h}^{-i c h}\right)}{h} I\right) .
\end{align}
Choosing $N=m-1$, $c=(m-1)/2$, $\alpha_{j}=(1-e_{h}^{i j h})/h$,\, $j=0,1, \ldots,m-1$ in (\ref{E6}) and multiplying both sides of
(\ref{E6}) by  $e_{h}^{-i\binom{m}{2} h}$, we obtain using (\ref{E1}) and (\ref{E2}) the operator relation (\ref{E3}). 

\end{proof}

\begin{lemma}
	\begin{align}\label{relopintegral}
		e_{h}^{-i(m-1) y} M_m f(y-(m-1)h)=e_{h}^{i(m-1) h}\Delta_{h}(e_{h}^{-i(m-1) y} M_{m-1} f(y-(m-1)h)).
	\end{align}
\end{lemma}
\begin{proof}
	We can verify (\ref{relopintegral}) as follows. By the $h$-product rule (\ref{derivpr}) with $f(x)=e^{-icx}_h$, $c=$constant,
	\begin{align*}
		\Delta_{h}(e^{-icx}_hg(x))&=g(x)(e^{-ic(x+h)}_h-e^{-icx}_h)/h  +e^{-ic(x+h)}_h \Delta_{h}g(x)\\&=e^{-ic(x+h)}_h\left(\Delta_{h}-(e^{ich}_h-1)/h\right)g(x).
	\end{align*}
	Choosing $c=m-1$, $g(x)=M_{m-1} f(x-(m-1)h)$, setting $x=y$, multiplying both sides by $e^{i(m-1)h}_h$, and using  definition  (\ref{E2}) of $M_m$ yields (\ref{relopintegral}).
\end{proof}

\begin{theorem}[\textbf{Integral representation for the $h$-exponential divided differences}] Let $(x_j-x_0)/h \in\mathbb{N}$, \, $j=1,2,\ldots m$. Then
	\begin{align}\label{intrep}
		\left[x_{0}, \ldots, x_{m};h\right]_e f=\frac{1}{A^{m-1}_h}\int_{x_{0}}^{x_{m}} e^{-i(m-1) y}_h E_{0, m}(y;h) M_{m} f(y-(m-1)h) d_h y, \quad m\geq 1,
	\end{align}
	where
	\begin{align}\label{Akh}
		A^{k}_h=	\left\{\begin{array}{ll}		1,	& \mbox{if}\quad k=0\\
			\left(e_{h}^{i h}/h\right)^{k}\left(1-e_{h}^{-i h}\right)\left(1-e_{h}^{-2i h}\right)\cdots\left(1-e_{h}^{-ki h}\right),	& \mbox{if}\quad k\neq 0.
		\end{array}\right.
	\end{align}
\end{theorem}

\begin{proof}
We use induction on $m$. First we  show that (\ref{intrep}) is true for $m=1$. In this case
\begin{align*}
	\frac{1}{A^{0}_h}\int_{x_{0}}^{x_{1}} E_{0, 1}(y;h) M_{1} f(y) d_h y&=\int_{x_{0}}^{x_{1}} \frac{1}{e_{h}^{ix_1}-e_{h}^{ix_0}}\Delta_h f(y) d_h y=\frac{1}{e_{h}^{ix_1}-e_{h}^{ix_0}}\int_{x_{0}}^{x_{1}}\Delta_h f(y) d_h y\\&=\frac{1}{e_{h}^{ix_1}-e_{h}^{ix_0}}\left(f(x_1)-f(x_0)\right)
	=\left[x_{0}, x_{1};h\right]_e f.
\end{align*}
Assume that (\ref{intrep})  is true for some $m\geq 1$. We will show that (\ref{intrep}) is true for $m + 1$. By (\ref{relopintegral}) the right-hand side of (\ref{intrep}) for $m+1$ multiplied by $A^{m}_h$ is equal to
\begin{align*}
	\int_{x_{0}}^{x_{m+1}} e^{-i m y}_h E_{0, m+1}(y;h) M_{m+1} f(y-m h) d_h y=& \int_{x_{0}}^{x_{m+1}}E_{0, m+1}(y;h)e_{h}^{im h}\Delta_{h}(e_{h}^{-im y} M_{m} f(y-mh)) d_h y\\=&-\int_{x_{0}}^{x_{m+1}}e_{h}^{-im y} M_{m} f(y-(m-1)h)\Delta_{h}E_{0, m+1}(y;h) d_h y,
\end{align*}
after applying the $h$-integration by parts formula (\ref{hintpart}) and using $E_{0,m+1}(y;h)=0$ at $y=x_0,x_{m+1}$, which follows from (\ref{hexpbspline}).
Using the recurrence relation (\ref{derivativeE}) for $\Delta E_{0, m+1}$ yields
\begin{align*}
	& \int_{x_{0}}^{x_{m+1}} e^{-i m y}_h E_{0, m+1}(y;h) M_{m+1} f(y-m h) d_h y\\
	&=-\int_{x_{0}}^{x_{m+1}}e_{h}^{-im y} M_{m} f(y-(m-1)h)    \frac{ \left(e^{-imh}_{h}-1\right)e^{i(y+h)}_{h} }{\left(e^{ix_{m+1}}_{h}-e^{ix_{0}}_{h}\right)h} \left(E_{1,m}(y;h)-	E_{0,m}(y;h) \right) d_h y\\&=		\frac{e^{ih}_h(1-e^{-imh}_h)}{(e^{ix_{m+1}}_{h}-e^{ix_{0}}_{h})h}\Bigg(\int_{x_{1}}^{x_{m+1}}e_{h}^{-i(m-1) y} M_{m} f(y-(m-1)h) E_{1,m}(y;h) d_h y\\&- \int_{x_{0}}^{x_{m}}e_{h}^{-i(m-1) y} M_{m} f(y-(m-1)h) E_{0,m}(y;h) d_h y\Bigg)\\&=\frac{A^{m}_h}{e^{ix_{m+1}}_{h}-e^{ix_{0}}_{h}}\Big(	\left[x_{1}, \ldots, x_{m+1} ; h\right]_{e} f-	\left[x_{0}, \ldots, x_{m} ; h\right]_{e} f\Big),
\end{align*}
where at the end we used the induction assumption and the relation $e^{ih}_h(1-e^{-imh}_h)A^{m-1}_h/h=A^{m}_h$. The result follows from the recurrence relation (\ref{recursion}) for the $h$-exponential divided difference. 	
\end{proof}
\begin{theorem}[\textbf{Integral representation for the $h$-trigonometric divided differences}]
	\begin{align}
		\left[x_{0}, \ldots, x_{m} ; h\right]_{t} f=(2 i)^{m-1} \frac{e_{h}^{-i (m-1)(m-4) h/4}}{A_{h}^{m-1}} \int_{x_0}^{x_{m}} T_{0, m}(y;h) L_{m} f(y-(m-1) h) d_{h} y,
	\end{align}
	where $A_{h}^{m-1}$ is given by (\ref{Akh}).
\end{theorem}
\begin{proof}
	This result is easily derived using (\ref{intrep}), (\ref{eq.relationbsplines}), (\ref{relation_te}) and (\ref{E3}).
\end{proof}

\section{Conclusions and Future Work}\label{section6}
We have shown that formulas and identities for classical exponential and classical trigonometric functions readily extend to $h$-exponential and $h$-trigonometric functions. Based on these extensions, we showed that many of the fundamental results for trigonometric B-splines, including the two-term recurrence relation, the two-term differentiation formula, and the Marsden identity, also readily extend to $h$-trigonometric B-splines. The corresponding formulas for classical trigonometric B-splines are actually just limiting cases of these formulas for $h$-trigonometric B-splines. In the future we hope to deepen our understanding of $h$-trigonometric
B-splines by developing a theory of $h$-trigonometric blossoming, as well as studying  knot insertion and
shape preserving properties for the $h$-variants of trigonometric B-splines.

The $q$-calculus is another variant of the quantum calculus. Can the theory of trigonometric B-splines be extended to a theory of $q$-trigonometric B-splines? This is yet another intriguing open problem for
future research.

\section*{Acknowledgments}

The first author  was supported by a grant (B\.{I}DEB-2219) from T\"{U}B\.{I}TAK, Scientific and Technological Research Council of Turkiye.


\begin{thebibliography}{00}
	
	
	\bibitem{victor}
Kac, V. G.,  Cheung, P. (2002). Quantum Calculus (Vol. 113). New York: Springer.

\bibitem{LycheThesis}

Lyche T., Discrete polynomial spline approximation methods (Ph.D. thesis), University of Texas, Austin, 1975.


\bibitem{Lyche survey}
Lyche, T., Trigonometric Splines; a survey with new results, in Shape Preserving Representations in Computer Aided Geometric Design, J. M. Pe\~{n}a (ed), Nova Science Publishers, Inc., New York, 1999, 201-- 227.

\bibitem{Lyche1}
Lyche, T.  and   Winther, R. (1979). A stable recurrence relation for trigonometric B-splines, J. Approx. Theory, 25(3): 266--279.

\bibitem{Schumaker2}
Schoenberg, I. J. (1964). On trigonometric spline interpolation. J. Math. Mech. 795--825.

\bibitem{Schumaker3}
Schumaker,  L. L. (1983). On hyperbolic splines,  J. Approx. Theory, 38(2): 144--166.


\bibitem{schumaker1}
Schumaker,  L. L. (1976). On Tchebycheffian spline functions, J. Approx. Theory, 18: 278--303.

\bibitem{ron-h}
Simeonov, P., Zafiris, V., and Goldman, R. (2011). $h$-Blossoming: A new approach to algorithms and identities for $h$-Bernstein bases and $h$-B\'{e}zier curves, Comput. Aided Geom. Design, 28(9): 549--565.





		\bibitem{fatma2}
	Z{\"u}rnaci-Yeti\c{s}, F., Non-polynomial B-spline functions (Ph.D. thesis), Dokuz
	Eyl{\"u}l University, Izmir, Turkiye (August 2022).
	
	\bibitem{fatma}
	{Z\"{u}rnac{\i}}, F., {Di\c{s}ib\"{u}y\"{u}k}, {\c{C}}. (2019). Non-polynomial divided
	differences and {B-spline} functions, J. Comput. Appl. Math. 349. 
	579--292.
	
		\bibitem{fatma3}
	Z{\"u}rnaci-Yeti\c{s}, F., {Di\c{s}ib\"{u}y\"{u}k},  {\c{C}}. (2025). Generalized Taylor series and Peano kernel theorem, Math. Methods Appl.  Sci. 48(5): 5521--5530. 
	
\end{thebibliography}


\end{document}